\documentclass[runningheads]{llncs}
\usepackage{amsmath, amssymb, lineno, lscape,hyperref,subcaption, verbatimbox, bm, enumitem}
\usepackage[T1]{fontenc}
\usepackage{graphicx}
\usepackage{color}

\begin{document}
\title{The need for accuracy and smoothness in numerical simulations}
%
%
\author{Carl Christian Kjelgaard Mikkelsen \inst{1}\orcidID{0000-0002-9158-1941 } \and
  Lori{\'e}n L{\'o}pez-Villellas \inst{2}\orcidID{0000-0002-1891-4359}
}
\authorrunning{C. C. Kjelgaard Mikkelsen et al. }
%
\institute{
  Department of Computing Science, Ume{\aa} University, 90187 Ume{\aa}, Sweden \email{spock@cs.umu.se} \\
  \and Departamento de Inform\'atica e Ingenier\'ia de Sistemas / Arag\'on Institute for Engineering Research (I3A), Universidad de Zaragoza, Zaragoza, Spain \\
  \email{lorien.lopez@unizar.es} \\
}

\maketitle 
\begin{abstract} We consider the problem of estimating the error when solving a system of differential algebraic equations. Richardson extrapolation is a classical technique that can be used to judge when computational errors are irrelevant and estimate the discretization error. We have simulated molecular dynamics with constraints using the GROMACS library and found that the output is not always amenable to Richardson extrapolation. We derive and illustrate Richardson extrapolation using a variety of numerical experiments. We identify two necessary conditions that are not always satisfied by the GROMACS library.\footnote{This is a preprint of a paper accepted by the conference PPAM-2024. The paper is to appear in Springer's LNCS series.}
\end{abstract}
\section{Motivation} Consider the problem of simulating the motion of a system of atoms moving in a force field subject to a set of constraints. In this case, Newton's 2nd law takes the form of the following system of differential algebraic equations
\begin{align}
  \bm{q}'(t) &= \bm{v}(t), \\
  \bm{M}\bm{v}'(t) &= \bm{f}(\bm{q}(t)) - \bm{G}(\bm{q}(t))^T\bm{\lambda}(t), \\
  \bm{g}(\bm{q}(t)) &= \bm{0}. 
\end{align}
The vector $\bm{q}$ represents the position of the atoms.
The vector $\bm{v}$ represents the velocities of the atoms. The function $\bm{f}$ represents the force acting on the atoms.
The nonsingular diagonal matrix $\bm{M}$ lists the masses of the atoms.
The function $\bm{G}$ is the Jacobian of the constraint function $\bm{g}$ and $\bm{\lambda}$ is a vector of Lagrange multipliers.
In the field of molecular dynamics, the standard algorithm for this problem is the SHAKE algorithm \cite{shake1977}.
It uses a pair of staggered grids with uniform step size $h$ and takes the form
\begin{align}
  \bm{v}_{n+1/2} &= \bm{v}_{n-1/2} + \bm{h} \bm{M}^{-1} \left( \bm{f}(\bm{q}_n) - \bm{G}(\bm{q}_n)^T \bm{\lambda}_n \right), \\
  \bm{q}_{n+1} &= \bm{q}_n + h \bm{v}_{n + 1/2}, \\
  \bm{g}(\bm{q}_{n+1}) &= \bm{0}. \label{equ:constraint}
\end{align}
The constraint equation \eqref{equ:constraint} is usually a nonlinear equation with respect to the Lagrange multipliers $\bm{\lambda}_n$. 
Now let $T \in \mathbb{R}$ denote any target value that can be computed in terms of the trajectory $t \rightarrow (q(t), v(t))$ and let $A_h$ denote the corresponding value obtained from the output of the SHAKE algorithm. It is clear that $T$ and $A_h$ are both functions of the force field $\bm{f}$ and the question of adjusting $\bm{f}$ to match the outcome of a physical experiment naturally suggests itself. Let therefore $T_0 \in \mathbb{R}$ be given and consider the problem of solving the equation
\begin{equation}
  T_0 = T(\bm{f})
\end{equation}
with respect to $\bm{f}$. The fundamental problem is that we cannot compute the exact values of $T(\bm{f})$ and $A_h(\bm{f})$. We must contend with the fact that we cannot expect to solve the constraint equations exactly nor can we avoid rounding errors in general. Let $\hat{A}_h(\bm{f}, \tau, u)$ denote the value returned by our computer when solving the constraint equations with a relative error bounded by $\tau$ and using floating point arithmetic with unit roundoff $u$. Suppose that $T_0 - \hat{A}_h(\bm{f}, \tau, u)$ is small. Can we conclude that $T_0 - T(\bm{f})$ is small? The triangle inequality delivers the following bound:
\begin{equation}
  |T_0 - T(\bm{f}) | \leq  |T_0 - \hat{A}_h(\bm{f},\tau,u)| + |A_h(\bm{f}) - \hat{A}_h(\bm{f},\tau,u)| + |T(\bm{f}) - A_h(\bm{f})|.
\end{equation}
We conclude that $T_0 - T(\bm{f})$ is small, if the \emph{computational error} $A_h(\bm{f}) - \hat{A}_h(\bm{f}, \tau, u)$ and the \emph{discretization error} $T(\bm{f}) - A_h(\bm{f})$ are both small. If we cannot control these errors, then we cannot say with certainty that our model delivers a good approximation of the physical reality.

Richardson extrapolation is a classical technique that is widely used in computational science and engineering applications \cite{roache1998}. It can be used to estimate the size of discretization errors or improve the accuracy of an existing solution \cite{zahari2018} and it has applications in event location \cite{mannshardt1978one} as well. As we shall demonstrate, Richardson extrapolation can often be used to determine when computational errors are insignificant.

In this paper we derive and illustrate the use of Richardson extrapolation using a variety of numerical experiments.
GROMACS is a state-of-the-art library for molecular dynamics that is widely used in academia \cite{gromacs2005}.
We demonstrate that the output of GROMACS is not always amenable to Richardson extrapolation. 
We identify two conditions that are not always satisfied by GROMACS and we demonstrate that each condition is necessary for the successful application of Richardson extrapolation.
Our data and software are freely available from our GitHub \cite{spockcc2024ppam} repository along with every script and function needed to replicate every number, table and figure from scratch. The names of our MATLAB functions are written with a typewriter font, e.g., \texttt{plot\_shells}.

\section{Theory}

Consider the problem of approximating a target value $T$ using a method $A = A_h$ that depends on a single real parameter $h$. We shall assume that there exists nonzero real constants $\alpha$ and $\beta$ and real exponents
\begin{equation}
  0 < p < q < r
\end{equation}
such that the error $E_h = T - A_h$ satisfies
\begin{equation} \label{equ:aex}
  E_h  = \alpha h^p + \beta h^q + O(h^r), \quad h \rightarrow 0_+.
\end{equation}
We say that the error $E_h$ satisfies an asymptotic error expansion. Frequently, the exponents $(p,q,r)$ are all integers, but since we shall encounter exponents that are not integers, we insist that $h$ is strictly positive.

Our first task is to estimate the error $E_h$ for a specific value of $h$. Richardson's error estimate $R_h$ is defined by the equation
\begin{equation}
 R_h =  \frac{A_h - A_{2h}}{2^p - 1}.
\end{equation}
The following theorem shows that Richardson's error estimate is a good approximation of the error when $h$ is sufficiently small.

\begin{theorem} If $E_h$ satisfies equation \eqref{equ:aex}, then
  \begin{equation}
    \frac{E_h - R_h}{h^q} \rightarrow  \left(1 - \frac{2^q-1}{2^p-1} \right) \beta, \quad h \rightarrow 0_+.
  \end{equation}
\end{theorem}
\begin{proof}
  By assumption, there is a function $h \rightarrow g(h)$ 
  \begin{equation}
    T - A_h  = \alpha h^p + \beta h^q + g(h)
  \end{equation}
  as well as constants $C>0$ and $h_0 > 0$ such that
  \begin{equation}
    \forall h \leq h_0 \: : \: |g(h)| \leq Ch^r.
  \end{equation}
  It follows that
  \begin{equation}
    T- A_{2h} = 2^p \alpha h^p + 2^q \beta h^q + g(2h).
  \end{equation}
  We conclude that
  \begin{equation} \label{equ:Dh:1}
    A_h - A_{2h} = (2^p - 1) \alpha h^p + (2^q - 1) \beta h^q + g(2h) - g(h).
  \end{equation}
  It follows that
  \begin{equation}
    R_h = \frac{A_h - A_{2h}}{2^p - 1} = \alpha h^p + \frac{2^q-1}{2^p-1} \beta h^q + \frac{g(2h)-g(h)}{2^p-1}.
  \end{equation}
  This implies that
  \begin{equation}
    \alpha h^p =  R_h - \frac{2^q-1}{2^p-1} \beta h^q + O(h^r).
  \end{equation}
  We conclude that
  \begin{equation}
    E_h = R_h + \left(1 - \frac{2^q-1}{2^p-1} \right) \beta h^q + O(h^r).
  \end{equation}
  The theorem follows immediately from this expression because $q < r$.
\end{proof}
We shall now demonstrate how the values of $p$ and $q$ can be determined by observing the exact values of $A_h$ for different values of $h$.
We define Richardson's fraction $F_h$ using the expression
\begin{equation}
  F_h = \frac{A_{2h} - A_{4h}}{A_h - A_{2h}}.
\end{equation}
The behavior of the function $h \rightarrow F_h$ is described by the following theorem.
\begin{theorem} If $E_h$ satisfies equation \eqref{equ:aex} and if $(m,n)$ is given by
  \begin{equation}
    m = q - p, \quad n = r - p,
  \end{equation}
  then Richardson's fraction $F_h$ satisfies
  \begin{equation}
    F_h \rightarrow 2^p, \quad h \rightarrow 0_+
  \end{equation}
  and 
  \begin{equation}
    \frac{F_h - 2^p}{h^m} \rightarrow (2^m-1) \nu, \quad \nu = \frac{2^q-1}{2^p-1} \frac{\beta}{\alpha}.
    \end{equation}
\end{theorem}

\begin{proof} It is convenient to rewrite equation \eqref{equ:Dh:1} as
  \begin{equation}
    A_h - A_{2h} = (2^p-1) \alpha h^p \Big [ 1  + \nu h^m + \phi(h) \Big]
  \end{equation}
  where $\phi(h) \in O(h^n)$. It follows immediately that
   \begin{equation}
      A_{2h} - A_{4h} = 2^p (2^p-1) \alpha h^p \Big[ 1 + 2^m \nu h^m + \phi(2h) \Big].
  \end{equation}
  This allows us to write
  \begin{equation}
    F_h = \frac{ A_{2h} - A_{4h}}{ A_h - A_{2h}} = 2^p \left[ \frac{ 1 + 2^m \nu h^m + \phi(2h)}{ 1  + \nu h^m + \phi(h)} \right].
  \end{equation}
  The fraction on the right-hand side is of the form
  \begin{equation}
    \frac{1 + f(h)}{1 + g(h)} = 1 + \frac{f(h) - g(h)}{1 + g(h)}
  \end{equation}
  where
  \begin{equation}
    f(h) = 2^m \nu h^m + \phi(2h), \quad g(h) = \nu h^m + \phi(h).
  \end{equation}
  It follows immediately that
  \begin{equation}
    F_h = 2^p \left( 1 + \frac{(2^m-1)\nu h^m}{1+g(h)} + \frac{\phi(2h) - \phi(h)}{1+g(h)}\right) \rightarrow 2^p, \quad h \rightarrow 0_+
  \end{equation}
  and
  \begin{equation}
    \frac{F_h - 2^p}{h^m}  = \frac{(2^m-1)\nu }{1+g(h)} + \frac{\phi(h) - \phi(2h)}{(1+g(h))h^m} \rightarrow (2^m - 1) \nu, \quad h \rightarrow 0_+,
  \end{equation}
  because $m < n$, so that
  \begin{equation}
     \frac{\phi(h) - \phi(2h)}{h^m} \rightarrow 0, \quad h \rightarrow 0_+.
  \end{equation}
  This completes the proof.
\end{proof}
We conclude that if $T-A_h$ satisfies the asymptotic error expansion \eqref{equ:aex}, then the order of the primary error term can be determined from the limit
\begin{equation}
2^p = \underset{h \rightarrow 0_+}{\lim} F_h
\end{equation}
and the difference $m = q-p$ can be determined from the fact that
\begin{equation} \label{equ:asymtotic-behavior-of-Fh}
  \log |F_h - 2^p| \approx \log(2^m-1) + \log|\nu| + m \log(h)
\end{equation}
is a good approximation for $h$ sufficiently small. In particular, we note that the right-hand side of equation \ref{equ:asymtotic-behavior-of-Fh} is a linear function of $\log(h)$ with slope $m$.
\section{Elementary examples}

The theory applies to the difference $T - A_h$ between the target value $T$ and the exact value of the approximation $A_h$.
In practice, the computed value $\hat{A}_h$ is different from the exact value $A_h$.
However, as we shall demonstrate shortly, is often possible to assert that the computational error is irrelevant and estimate the error $T - \hat{A}_h$ accurately. We begin by considering the familiar problem of computing definite integrals
\begin{equation} \label{equ:integral}
  T = \int_{a}^b f(x) dx
\end{equation}
using the composite trapezoidal rule $A_h$ given by
\begin{equation}
  A_h = \frac{1}{2}h \sum_{j=0}^{n-1} \left[ f(x_j) + f(x_{j+1}) \right], \quad x_j = jh, \quad nh = b-a, \quad n \in \mathbb{N}.
\end{equation}
It is well-known that if $f \in C^{\infty}([a,b], \mathbb{R})$, then there exists a sequence $\{\alpha_j\}_{j=1}^\infty \subset \mathbb{R}$ such that 
\begin{equation} \label{equ:trapezoidal-rule:aex}
  E_h = \sum_{j=1}^k \alpha _j h^{2j} + O(h^{2k+1}), \quad h \rightarrow 0_+.
\end{equation}
In particular, $(p,q,r) = (2,4,6)$ when $f$ is everywhere smooth.
\paragraph{Integration of a function that is everywhere smooth.}

Let $f : [0, 1] \rightarrow \mathbb{R}$ be given by $f(x) = e^x$ and $T$ be given by equation \eqref{equ:integral}.
The script {\tt rint\_mwe1} computes the composite trapezoidal sum $A_h$ using $h_k = 2^{-k}$ for $k \in \{0,1,\dots,19\}$ and generates \ref{fig:rint_mwe1a} and \ref{fig:rint_mwe1b}.
The raw data shows that $\hat{A}_{h_k}$ approaches $4 = 2^2$ as $k$ increases and $k \in \{2,3,\dots,14\}$.
This suggests that $p=2$.
Figure \ref{fig:rint_mwe1a} illustrates the evolution of the \emph{computed} values of Richardson's fraction.
We observe that $k \rightarrow \log_2|\hat{F}_{h_k} - 4|$ is essentially a linear function of $k$ with slope $-m = -2$ for $k \in \{2,3,\dots,10\}$.
This is the so-called asymptotic range, where the computed value $\hat{A}_h$  behaves in a manner that is indistinguishable from the exact value $A_h$.
We conclude that the experiment supports the existence of an asymptotic error expansion with $(p,q) = (p,p+m) = (2,4)$.
Since the target value $T$ is known, we can treat Richardson's error estimate as an approximation of the error $T - \hat{A}_h$ and compute the corresponding relative error, see Figure \ref{fig:rint_mwe1b}.
We observe that the computed value of Richardson's error estimate is a good approximation of the error $T - \hat{A}_{h_k}$. In fact, the corresponding relative error decreases when $k$ increases as long as we also remain inside the asymptotic region.

\begin{figure}[h]
\begin{subfigure}[h]{0.49\linewidth}
\includegraphics[width=\linewidth]{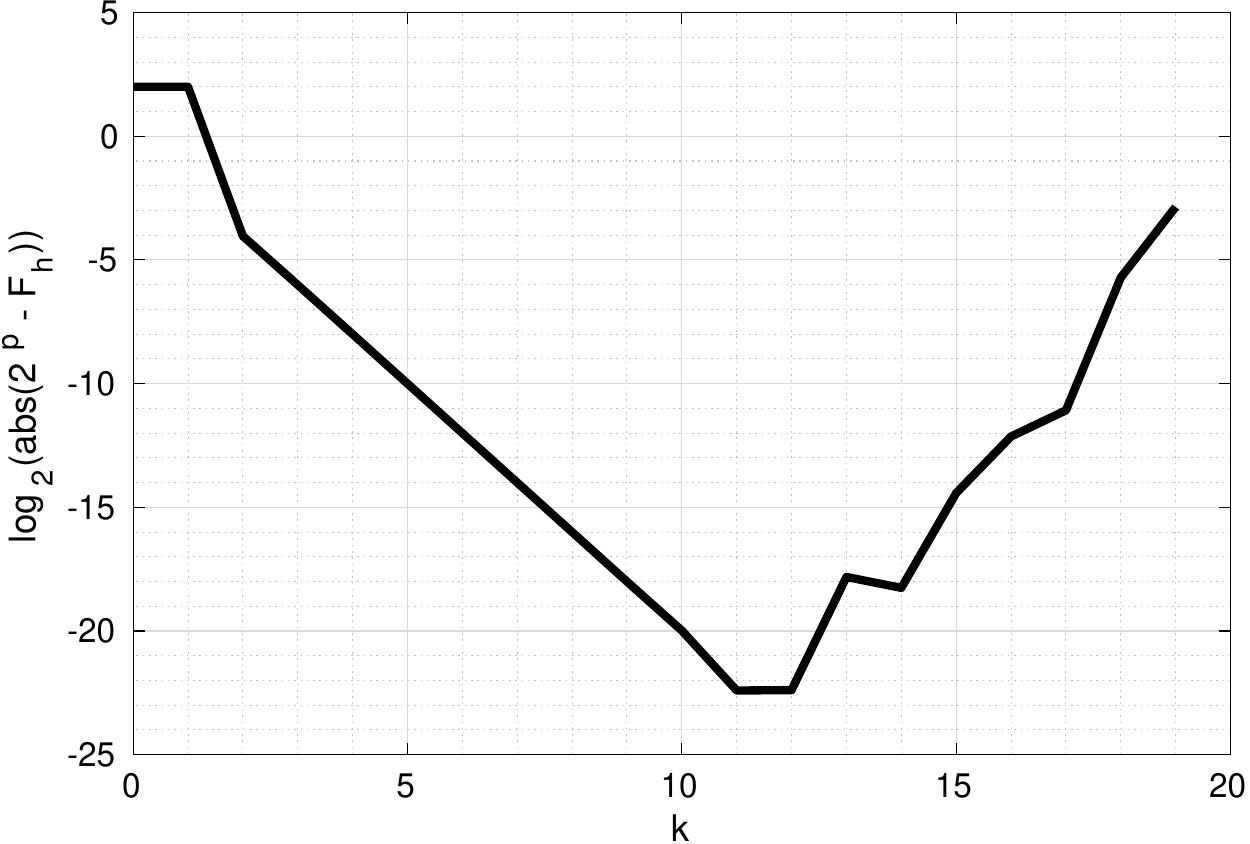} 
\caption{The evolution of $F_h$} \label{fig:rint_mwe1a}
\end{subfigure}
\hfill
\begin{subfigure}[h]{0.49\linewidth}
\includegraphics[width=\linewidth]{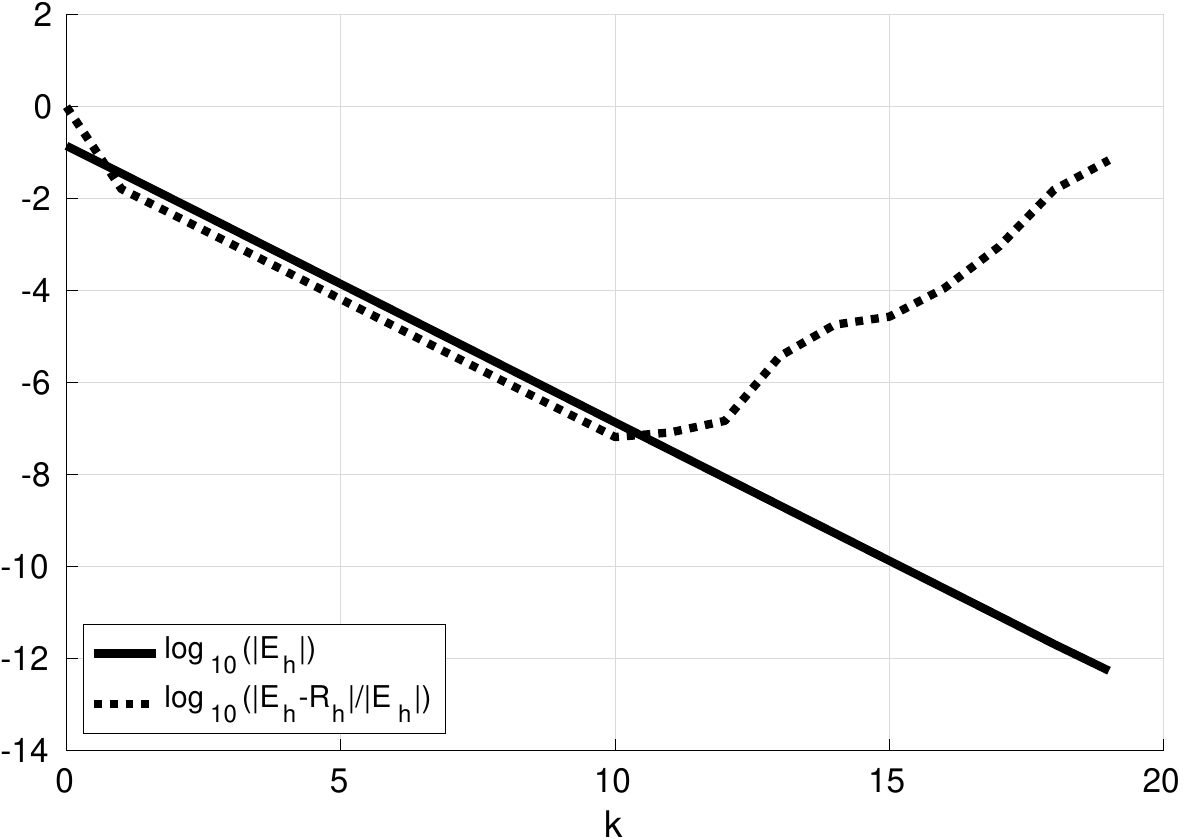} 
\caption{The size of $E_h$ and the accuracy of $R_h$} \label{fig:rint_mwe1b}
\end{subfigure}%
\caption{The behavior of $F_h$, $E_h$ and $R_h$ for a method with $(p,q) = (2,4)$.}
\end{figure}

\paragraph{Integration of a function that is smooth in all but one point.}

Let $f : [0,1] \rightarrow \mathbb{R}$ be given by $f(x) = \sqrt{x}$ and let $T$ be given by equation \eqref{equ:integral}.
Then $T = \frac{2}{3}$. Since $f$ is not differentiable at $x=0$ we have no guarantee that there exists an asymptotic error expansion of the form given by equation \eqref{equ:trapezoidal-rule:aex}.
The script {\tt rint\_mwe2} computes $A_h$ using $h_k = 2^{-k}$ for $k \in \{0,1,\dots,25\}$ and generates Figures \ref{fig:rint_mwe2a} and \ref{fig:rint_mwe2b}.
The raw data shows that $p = 2$ cannot be true, but it is plausible that $p \approx \frac{3}{2}$.
Figure \ref{fig:rint_mwe2a} illustrates the evolution of the \emph{computed} values of Richardson's fraction.
We observe that $k \rightarrow \log_2|\hat{F}_{h_k} - 2^{3/2}|$ is essentially a linear function of $k$ with slope $-\frac{1}{2}$ for $k \in \{2,3,\dots,18\}$.
This is the asymptotic range where the computed numbers $\hat{A}_h$ behave in a manner that is similar to the exact value $A_h$. 
We conclude that the experiment is consistent with an asymptotic error expansion with $(p,q)=(\tfrac{3}{2},2)$.
Since the target value $T$ is known, we can treat Richardson's error estimate as an approximation of the error $T-\hat{A}_h$ and compute the corresponding relative error, see Figure \ref{fig:rint_mwe2b}.
We observe that the computed value of Richardson's error estimate is a good approximation of the error $T - \hat{A}_h$.
In fact, the corresponding relative error decreases when $k$ increases and we remain inside the asymptotic region.

\begin{figure}
\begin{subfigure}[h]{0.49\linewidth}
\includegraphics[width=\linewidth]{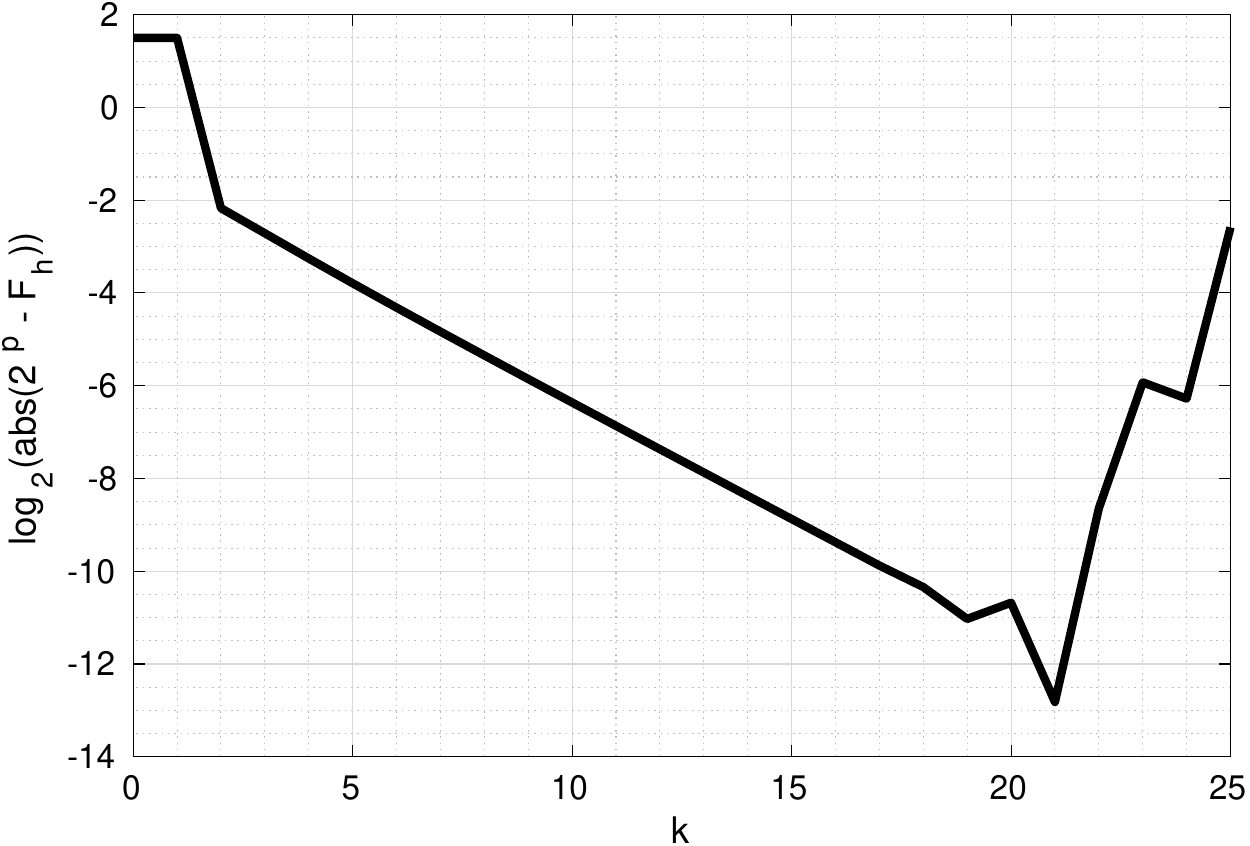} 
\caption{The evolution of $F_h$} \label{fig:rint_mwe2a}
\end{subfigure}
\hfill
\begin{subfigure}[h]{0.49\linewidth}
\includegraphics[width=\linewidth]{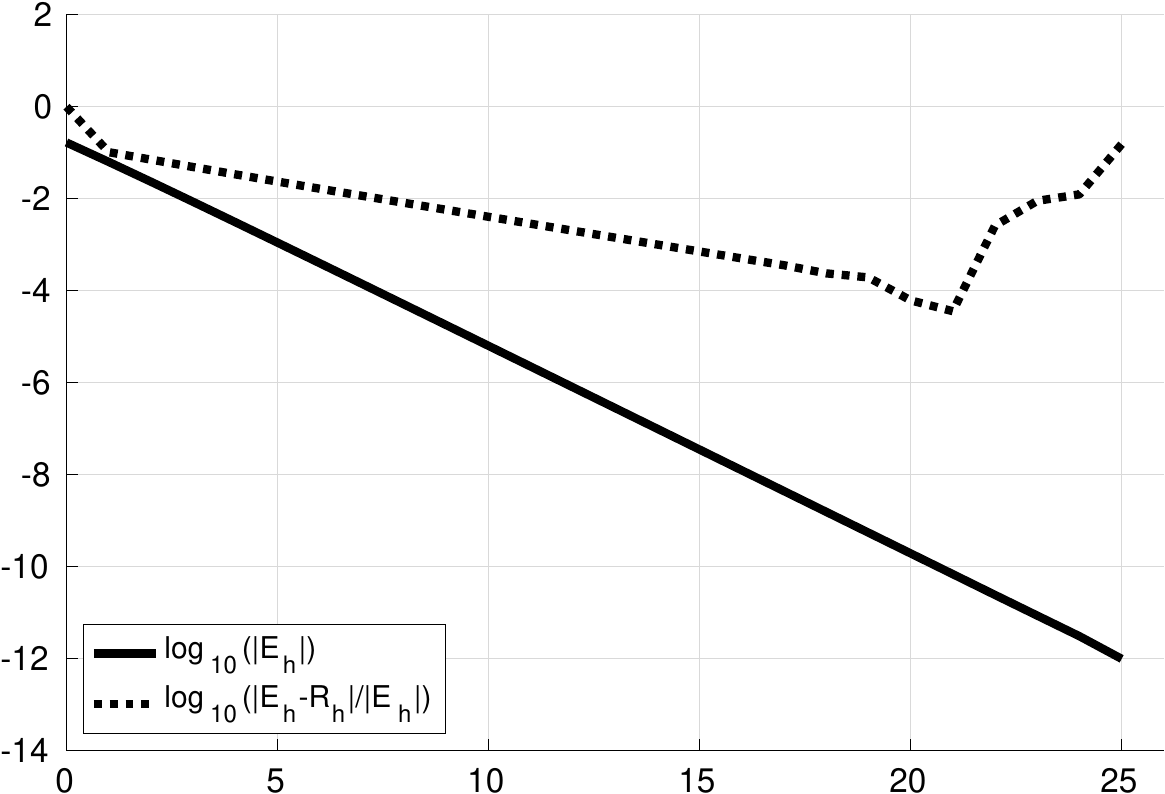} 
\caption{The size of $E_h$ and the accuracy of $R_h$.} \label{fig:rint_mwe2b}
\end{subfigure}%
\caption{The evolution $F_h$ and the accuracy of $R_h$ for a method with $(p,q) = (\frac{3}{2},2)$.}
\end{figure}

We mention in passing that low order methods are more practical than high order methods in the sense that low order methods tend to have asymptotic ranges that are larger than high order methods.
This is due to the fact the function $h \rightarrow F_h$ suffers from subtractive cancellation when $h$ is sufficiently small. This issue is more acute for high order methods than for low order methods, because $A_h$ tends to $T$ much more rapidly for high order methods than for low order methods.

\section{Practical examples}

In this section we present the results of more elaborate experiments that highlight both the utility and the practical limitations of Richardson's extrapolation.

\subsection{A successful application of the theory}


\paragraph{Example: Identify the shells fired by a howitzer.}

Consider the D-20 howitzer whose maximum range is known to be about 17.3 km \cite{foss1976artillery}. We have access to tables of the drag coefficient for 6 different shells types \cite{jbmballistics}. Can we determine the drag coefficient that provides the best match to the physical reality?

The script {\tt maxrange\_rk1} models a shell as a point particle moving in a plane subject to Earth's standard gravity and the international standard atmospheric model.
Each trajectory is integrated using Euler's explicit method ({\tt 'rk1'}) and all but the final step has the same size $h$. The final step is adjusted to place the shell on the ground. The drag functions are interpolated from tables using cubic spline interpolation. The function {\tt plot\_shells} will plot the drag coefficients for the different shells as a function of the Mach number.
For each drag coefficient our target value $T$ is the maximum range of the shell as the elevation of the howitzer varies continuously from $0$ to $\frac{\pi}{2}$. For each drag coefficient, we compute $12$ different approximations $A_{h_k}$ of $T$ using the step size $h_k = 2^{3-k}$ seconds, where $k \in \{1,2,\dots,12\}$.
For each drag coefficient and for each value of the time step $h$, a range function is defined which returns the range of the shell as a function of the howitzers elevation $\theta$.
The range functions are unimodal and the maximum range is found using the golden section search algorithm.
The initial search bracket is $[0,\pi/2$] and this bracket is systematically reduced in length until it is shorter than $\frac{\pi}{2}u$.
The script will either read the raw data from a file or generate it from scratch.
In any case, the script produces several figures and tables including Table \ref{tab:maxrange_rk1_table_tol53} and Figure \ref{fig:maxrange_rk1_fraction_tol53}.
\begin{table}
 \caption{The computed maximum range for 6 shells fired from a D-20 howitzer.}
 \label{tab:maxrange_rk1_table_tol53} 
 \centering
 \vspace{.25cm}
 \begin{tabular}{c|c|c}
   Shell type & Maximum range (m) & Error estimate (m) \\ \hline
   G1      &     12832    &            0.4 \\
   G2      &     16857    &            0.1 \\
   G5      &     15918    &            0.2 \\
   G6      &     15556    &            0.2 \\
   G7      &     17461    &            0.1 \\
   G8      &     15914    &            0.1 \\
   \hline
 \end{tabular}
\end{table}
These two figures represent calculations where the final time step is computed with an error that is bounded by $u h_{k}$ where $u = 2^{-53}$ is the double precision unit roundoff. Table \ref{tab:maxrange_rk1_table_tol53} lists the maximum range and the corresponding error estimate for each of the 6 shell types in our library using a time step of $h = 2^{-9} s$. In each case the error estimate suggests that the computed range is exact to the number of figures shown. In particular, we see that a G7 type shell achieves a maximum range of 17.5 km and all other shells have ranges that are less than 16.9 km. However, it is a fallacy to conclude anything on the basis of this table alone. In each case, we need to assert that we are inside the asymptotic range and that the error estimates are reliable. To this end, we examine the evolution of Richardson's fraction for the maximum range of each shell, see Figure \ref{fig:maxrange_rk1_fraction_tol53}. For each of the 3 different drag coefficients shown we see that the evolution of Richardson's fraction supports an asymptotic error expansion with $(p,q) = (1,2)$. This result is consistent with the use of Euler's explicit method which is 1st order accurate in the time step\footnote{The remaining figures are similar and have been omitted to save space.}. We observe that for each drag coefficient, $k=12$ is still inside the asymptotic range and we have no reason to doubt the magnitude of the error estimate. We conclude that the best model for the D-20 howitzer is in fact the G7 shell.

\begin{figure}[t!]
  \centering
  \includegraphics[width=\linewidth]{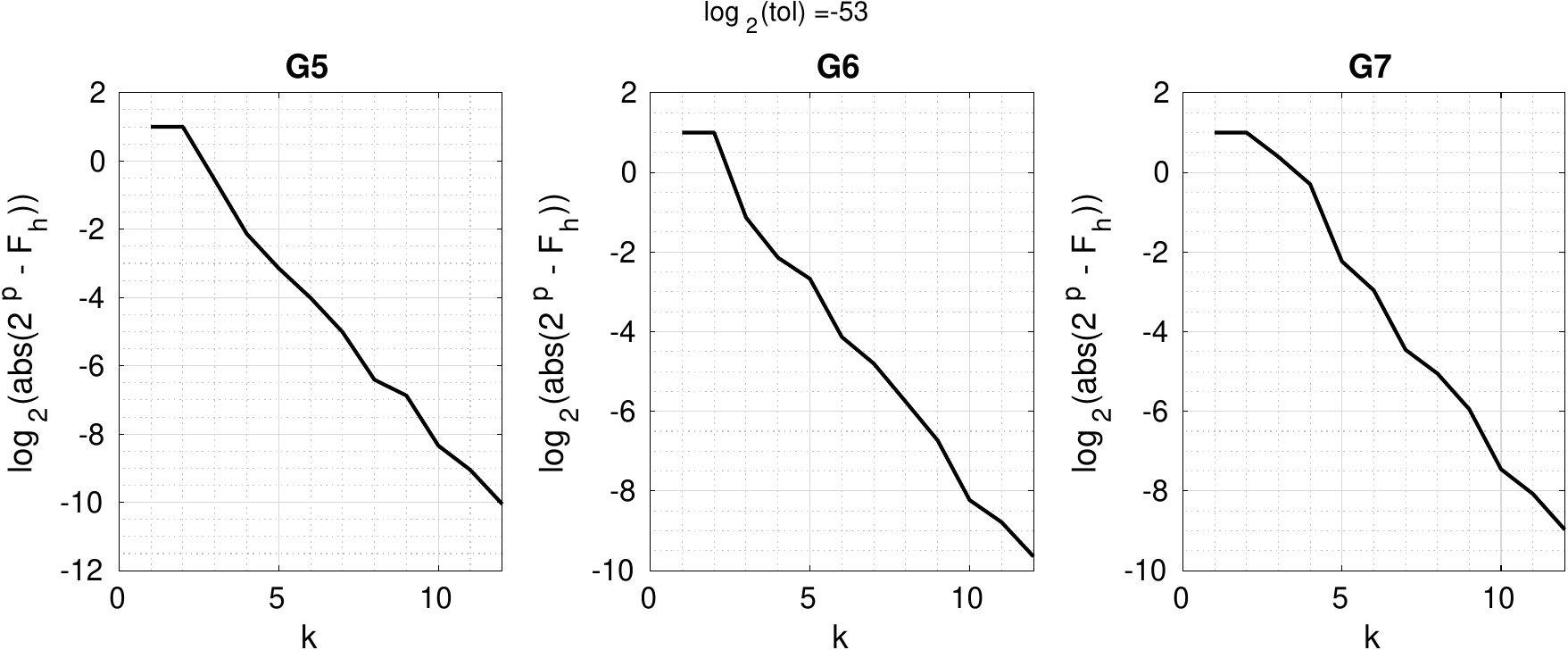} 
  \caption{The evolution of Richardson's fraction corresponding to the maximum range of 3 different shells fired from the D-20 howitzer.} \label{fig:maxrange_rk1_fraction_tol53}
\end{figure}

\subsection{An unsuccessful application of the theory}

We utilized GROMACS v2021 to conduct experiments on the behavior of hen egg white lysozyme submerged in water within a cubic simulation box, following Justin Lemkul's Lysozyme in Water GROMACS Tutorial \cite{lemkul2019from}. Several steps were taken to prepare the system for production simulation: first, ions were introduced to achieve electrical neutrality. Subsequently, energy minimization was performed using the steepest descent algorithm until the maximum force reached below 1000.0 kJ/(mol·nm). Following this, the system underwent 100 ps of equilibration in an NVT ensemble to stabilize temperature, followed by another 100 ps of equilibration in an NPT ensemble to stabilize pressure. The described process was replicated using two different force fields, OPLS-AA/L and CHARMM36. We conducted production simulations of 1 ps for both force fields, using $n \in \{250, 500, 1000, 1100:100:2000, 3000:1000:16000 \}$ steps to cover this interval.
Moreover, we used two different values of the tolerance $\tau$ for the SHAKE algorithm, namely $\tau \in \{10^{-4}, 10^{-12}\}$. 
For each experiment, we computed the total kinetic and potential energy of the system at the end of the simulation.
The function {\tt gromacs\_figures} will generate Figure~\ref{fig:oplsaaltol04} and Figure~\ref{fig:oplsaaltol12} for the OPLS-AA/L force field and similar figures for CHARMM36.

\begin{figure}
  \centering
  \includegraphics[width=\linewidth]{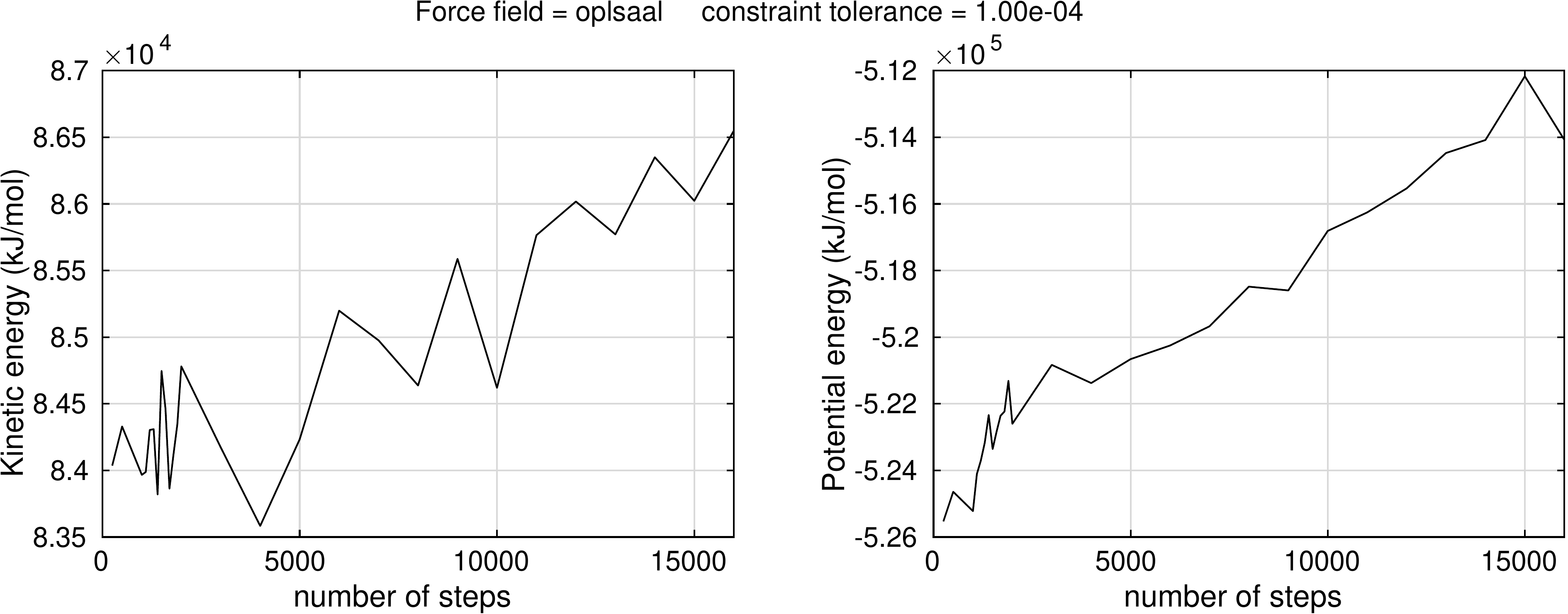}  
  \caption{The evolution of the kinetic and potential energy of a system as the number of time-steps used to cover $1$ ps of real time.} \label{fig:oplsaaltol04}
\end{figure}

These figures display the total potential and kinetic energy at the end of the simulation as a function of the total number $n$ of time steps used to cover the interval. The figures present several features of interest.
Firstly, the potential energy and especially the kinetic energy exhibits violent oscillations when the tolerance is large, i.e., $\tau = 10^{-4}$. The amplitude of the oscillations is reduced when $\tau = 10^{-12}$. We expect the solution of the underlying differential algebraic equation to behave nicely, but we have no such expectation for the computed approximation unless $\tau$ is very small. 
Secondly, the total energy grows linearly with the number of time steps. This is not surprising as we expect the rounding error to grow with the number of operations.
Thirdly, if the computed energies for ${\tt tol} = 10 ^{-12}$ follow an asymptotic error expansion, then the commonly used time step of $1$ fs ($n=1000$ in this case) is \emph{not} well inside the asymptotic range. Why is this? If we were in the asymptotic range, then $\hat{A}_h \approx T - \alpha h^p$ would be a good approximation for some $\alpha \not = 0$ and $p>0$. In particular, the value of $\hat{A}_h$ should behave in a \emph{monotone} manner and the tiny oscillations that we have recorded should not be present.

\begin{figure}
  \centering
  \includegraphics[width=\linewidth]{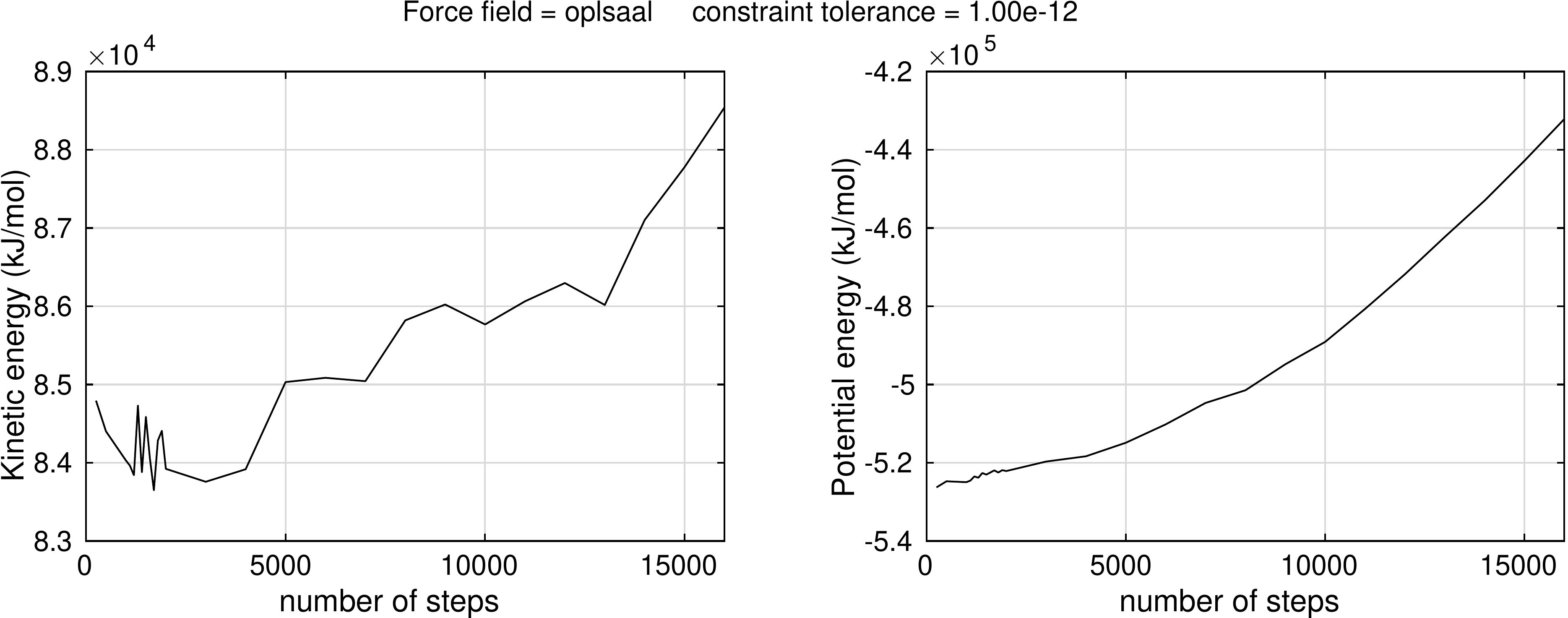} 
   \caption{The evolution of the kinetic and potential energy of a system as the number of time-steps used to cover $1$ ps of real time.} \label{fig:oplsaaltol12}
\end{figure}

\section{The difference between success and failure}

In this section we identify two conditions that are not always satisfied by the GROMACS library and we demonstrate that they are necessary for the successful application of Richardson extrapolation.

\subsection{The need for sufficient accuracy}

It is clear that the output of constrained MD simulation depends on the tolerance passed to the constraint solver. Similarly, when computing the range of a howitzer it is necessary to adjust the final time-step to place the shell on the ground with great accuracy. In the case of Euler's explicit method, the relevant equation is linear, but in general it is a nonlinear equation. The function {\tt maxrange\_rk1\_mwe1} and its companion {\tt maxrange\_rk2\_mwe2} both use the bisection method to compute the final time step with an error that is bounded by $h\cdot\text{tol}$ using a wide range of tolerance {\tt tol}.
\begin{figure}
  \includegraphics[width=\linewidth]{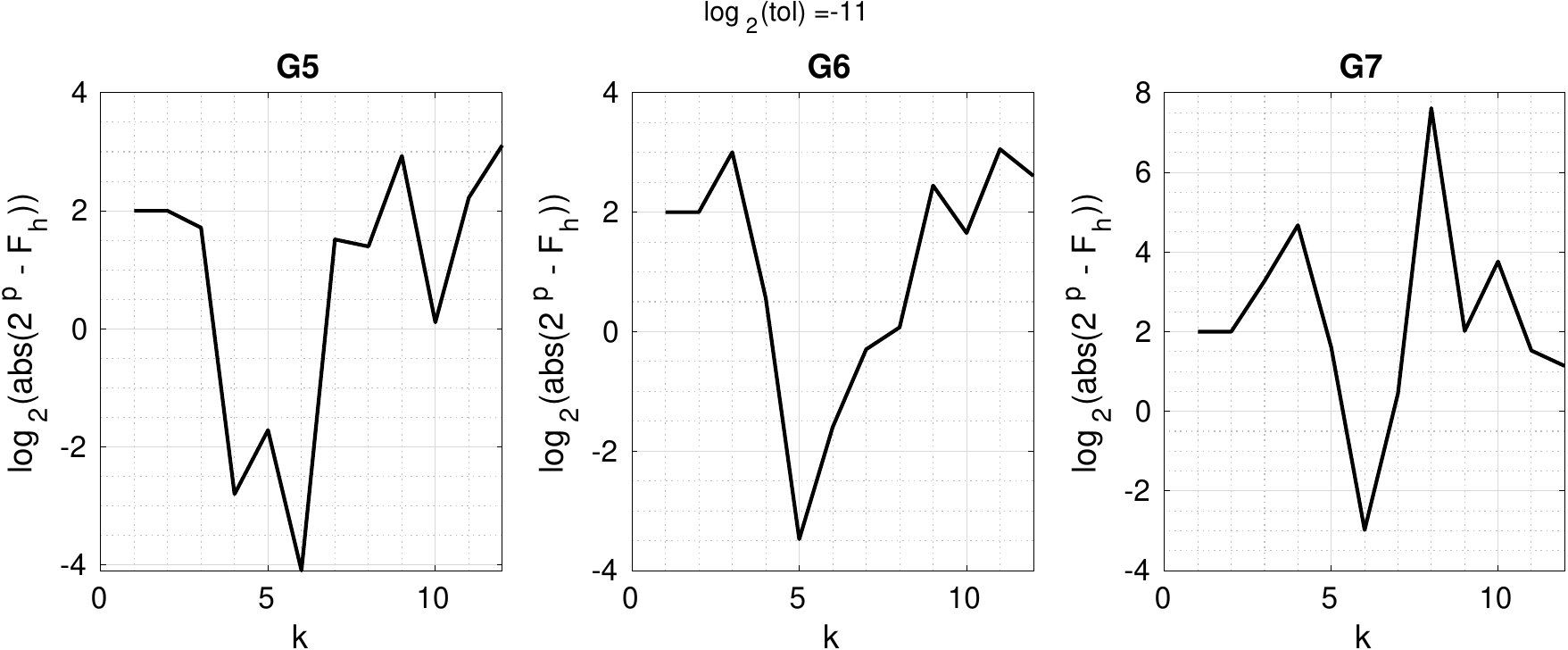} 
  \caption{The evolution of $F_h$ for the maximum range of 3 different shells using Heyn's method (\text{'rk2'}) and $\text{tol} = 2^{-11}$.}
  \label{fig:maxrange_rk2_tol11}
\end{figure}
It is instructive to observe the consequences of solving this equation inaccurately. Figures \ref{fig:maxrange_rk2_tol11} ($\text{tol} = 2^{-11}$) and Figures \ref{fig:maxrange_rk2_tol25} ($\text{tol} = 2^{-25}$) show the evolution of Richardson's fraction for two different values of the tolerance {\tt tol}.
\begin{figure}
  \includegraphics[width=\linewidth]{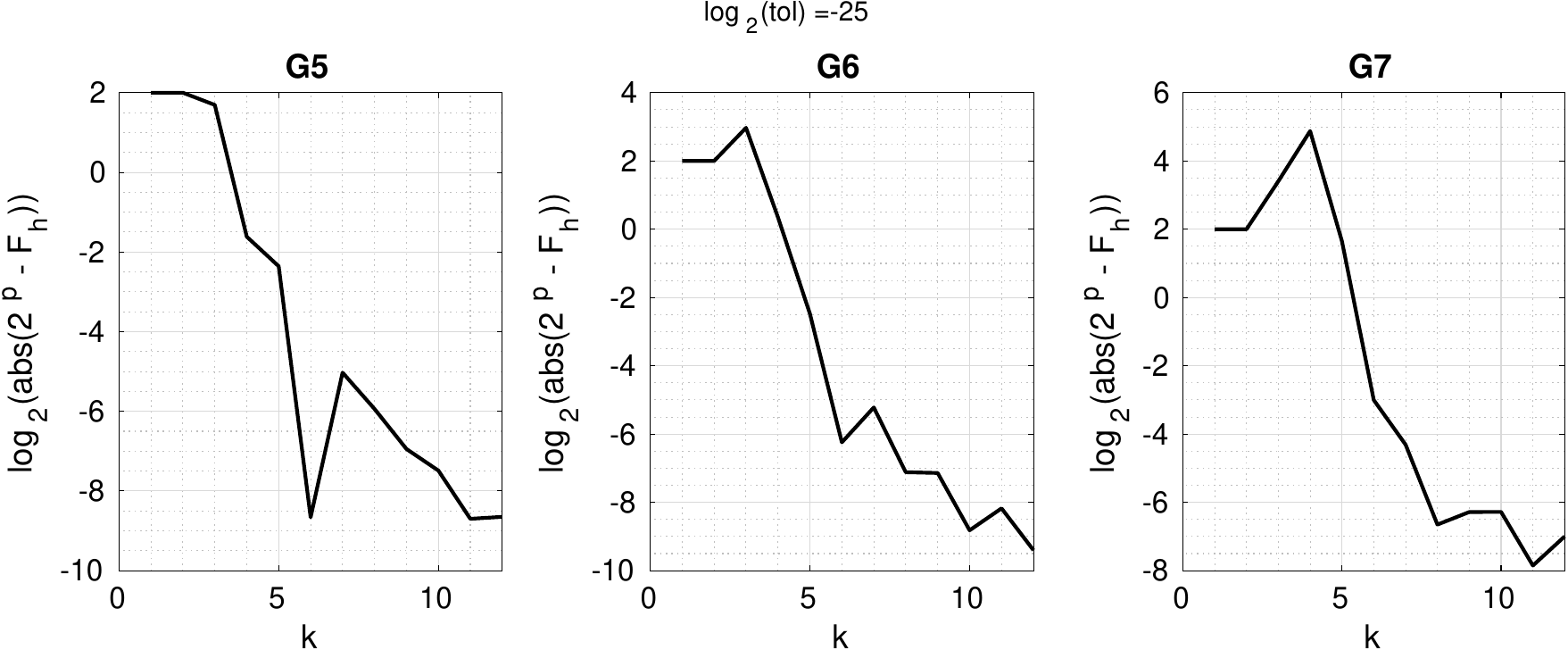} 
  \caption{The evolution of $F_h$ for the maximum range of 3 different shells computed using Heyn's method (\text{'rk2'}) and $\text{tol} = 2^{-25}$.}
  \label{fig:maxrange_rk2_tol25}
\end{figure}
When the tolerance is $\tau = 2^{-11}$, see \ref{fig:maxrange_rk2_tol11}, there is no evidence that an asymptotic error expansion exists and there is no reason to trust Richardson’s error estimate. When the tolerance is $\tau = 2^{-25}$, see Figure \ref{fig:maxrange_rk2_tol25}, the fact that $\hat{F}_h$ approaches $2^p$ suggests that an asymptotic error expansion exists, but it is not trivial to determine an asymptotic range.
Regardless, it is clear that if we do not know which tolerance is sufficient and if an error estimate is required, then our safest course of action is to solve all equations as accurately as the hardware will allow.

\subsection{The need for sufficient smoothness}

In molecular dynamics, it is common to ignore the interaction between atoms that are far away. This can be done by setting force fields to zero outside of a sufficiently large ball. There is more than one way to achieve this and the documentation for GROMACS 2021 discusses its use of force fields that are not of class $C^{\infty}$.

\begin{figure}[h!]
  \centering
  \includegraphics[width=\linewidth]{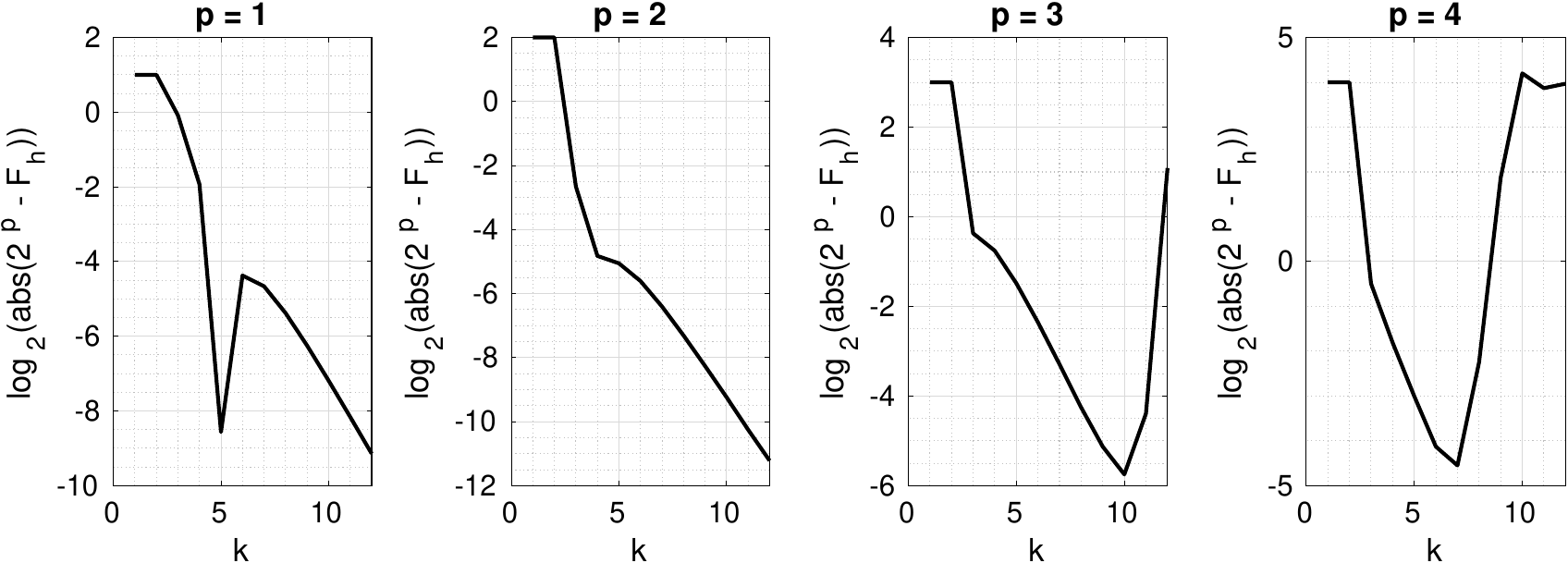} 
  \caption{The evolution of $F_h$ for the total kinetic energy of a system of ions using integrators of order $p \in \{1,2,3,4\}$.}
  \label{fig:iontrap_mwe1}
\end{figure}

In order to explore the importance of smoothness we have simulated the motion of a set of identical ions moving in a liquid. The ions repel each other electrostatically, but they are pulled towards the origin by independent and identical springs that obey Hooke's law. The friction between each ion and the liquid is proportional to its velocity. The friction drains the energy and ensures that the ions eventually come to rest in a stable configuration. 
Let $\bm{f}$ denote the force-field generated by an ion located at $0$ with charge $q$. Then
\begin{equation}
  \bm{f}(\bm{r}) = c q \bm{r}/r^3, \quad r = \|\bm{r}\|_2,
\end{equation}
where $c > 0$ is a suitable constant. The script {\tt iontrap\_mwe1} does not modify the electrostatic force fields and the $m=4$ ions ultimately form a regular tetrahedron with edge length $\rho > 0$.
The scripts {\tt iontrap\_mwe2} and {\tt iontrap\_mwe4} replace $\bm{f}$ with
\begin{equation}
\bm{f}_k(\bm{r}) = \bm{f}(\bm{r}) g_k(r), \quad k \in \{2, 4\}
\end{equation}
where $g_k$ is a switching function that assumes values in $[0,1]$. The function $g_2$ has a jump discontinuity and satisfies $g_2(r) = 1$ for $r < 0.5 \rho$ and $g_2(r) = 0$ for $r \ge 0.5 \rho$. The function $g_4$ is of class $C^\infty$ and satisfies $g_4(r) = 1$ for $r < 0.5 \rho$ and $g_4(r) = 0$ for $r \ge 0.95 \rho$. It is clear that changing the force fields impacts the motion, but can we estimate the discretization error and quantify this effect? Figure \ref{fig:iontrap_mwe1}, \ref{fig:iontrap_mwe2}, \ref{fig:iontrap_mwe4} show the evolution of Richardson's fraction for the kinetic energy at the end of each simulation.

When the force fields are not perturbed, see Figure~\ref{fig:iontrap_mwe1}, or when the perturbation is smooth, see Figure~\ref{fig:iontrap_mwe4}, then the experiments support the existence of an asymptotic error expansion and we can clearly identify an asymptotic range for each of the 4 Runge-Kutta methods used to integrate Newton's equations of motion.x1
\begin{figure}[h!]
  \centering
  \includegraphics[width=\linewidth]{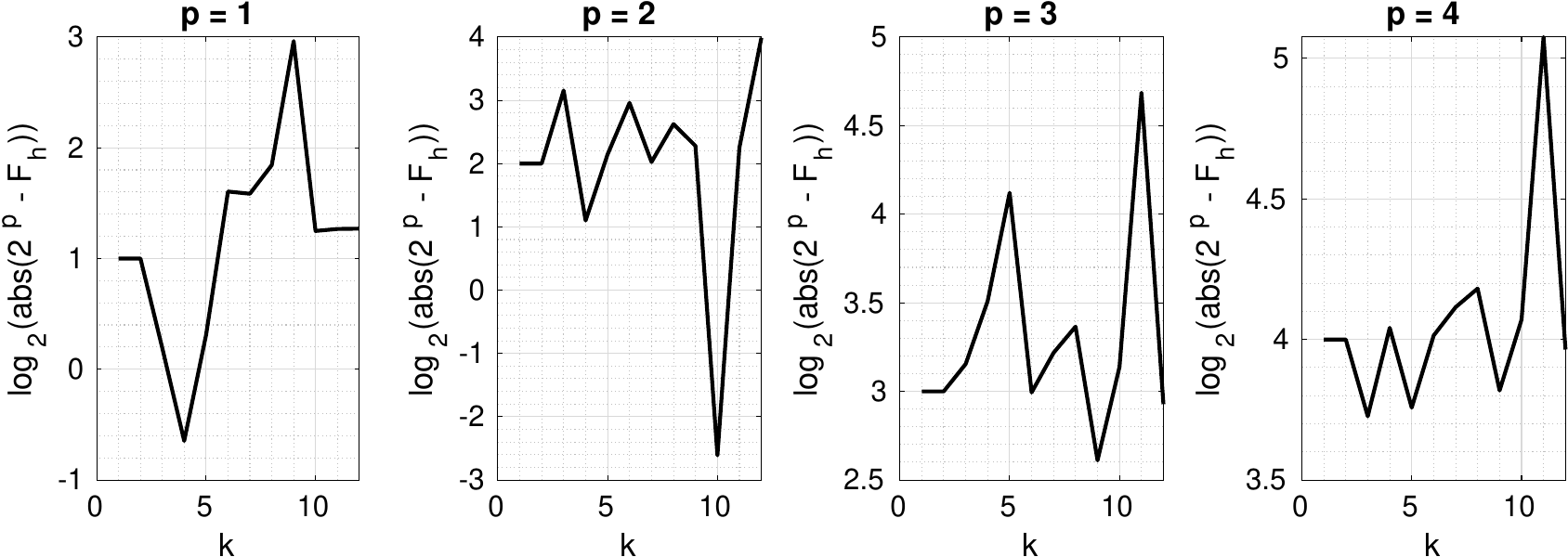} 
  \caption{The evolution of $F_h$ for the total kinetic energy of a system of ions using integrators of order $p$. The fields are zero outside a small sphere and discontinuous.}
  \label{fig:iontrap_mwe2}
\end{figure}
When the force fields are truncated and discontinuities are introduced into the simulation, see Figure~\ref{fig:iontrap_mwe2}, there is no evidence to support the existence of an asymptotic error expansion and there is no reason to suspect that Richardson's error estimate is accurate.
\begin{figure}
  \centering
  \includegraphics[width=\linewidth]{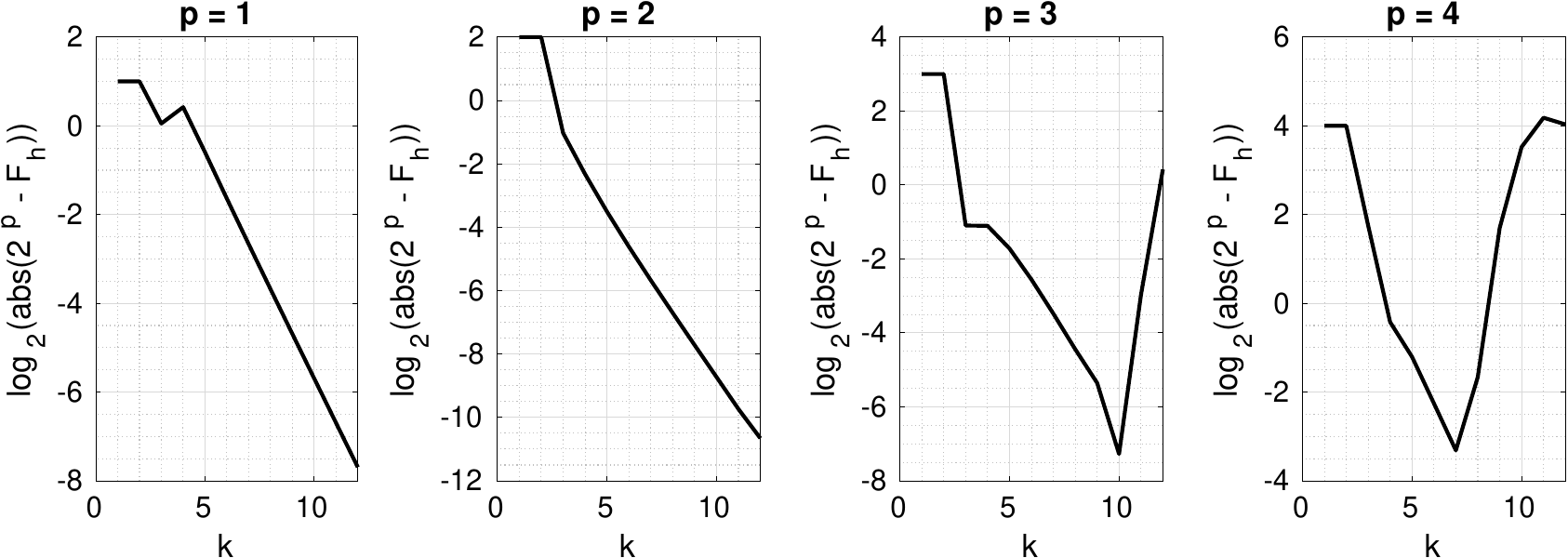} 
  \caption{The evolution of $F_h$ for the total kinetic energy of a system of ions using integrators of order $p$. The fields are zero outside a small sphere, but smooth.}
  \label{fig:iontrap_mwe4}
\end{figure}

\section{Conclusion}

A central task in the field of computational science is to fit a system of DAEs to the results of a physical experiment. It is often possible to assert when the computational errors are irrelevant, estimate the discretization error using Richardson extrapolation. This classical technique hinges on the existence of an asymptotic error expansion. However, if the functions that describe our problem are not sufficiently differentiable or if the computational error is not sufficiently small, then we cannot estimate the discretization error and we lose the ability to evaluate our model. In the absence of further analysis, the best strategy is therefore to use functions that are smooth and solve all central equations as accurately as possible.

\subsubsection{Acknowledgments}

We would like to thank the reviewers for their work which allowed us to improve the manuscript. We would like to thank Jesús Alastruey-Benedé, Pablo Ibáñez and Pablo García-Risueño for stimulating discussions on the subject matter.
The first author is supported by eSSENCE, a collaborative e-Science programme funded by the Swedish Research Council within the framework of the strategic research areas designated by the Swedish Government.
This work has been partially supported by the Spanish Ministry of Science and Innovation MCIN/AEI/10.13039/501100011033 (grant PID2022-136454NB-C22), and by Government of Aragon (T58\_23R research group).

%
%
%

\end{document}